\documentclass{amsart}
\usepackage{amssymb}
\usepackage{amsmath,amsthm,amscd,amssymb, mathrsfs}
\usepackage{amsfonts}
\usepackage{latexsym}
\usepackage{color}

\usepackage{footmisc}

\usepackage[colorlinks,citecolor=red,pagebackref,hypertexnames=false]{hyperref}

\numberwithin{equation}{section}

\theoremstyle{plain}
\newtheorem{theorem}{Theorem}[section]
\newtheorem{lemma}[theorem]{Lemma}

\newtheorem{proposition}[theorem]{Proposition}
\newtheorem{conjecture}[theorem]{Conjecture}
\newtheorem{question}[theorem]{Question}

\theoremstyle{definition}

\theoremstyle{remark}
\newtheorem{remark}[theorem]{Remark}

\newtheorem{case[theorem]}{Case}

\title[Magnitude of the sum of any $k$ points]{On the sums of any $k$ points in finite fields}
\author{David Covert, Doowon Koh*, and Youngjin Pi}

\address{Department of Mathematics and Computer Science\\
University of Missouri-St. Louis \\
St. Louis, MO 63121 USA} \email{covertdj@umsl.edu}

\address{Department of Mathematics\\
Chungbuk National University \\
Cheongju city, Chungbuk-Do 361-763 Korea}
\email{koh131@chungbuk.ac.kr}

\address{Department of Mathematics\\
Chungbuk National University \\
Cheongju city, Chungbuk-Do 361-763 Korea}
\email{pi@chungbuk.ac.kr}

\thanks{*Corresponding Author 
Doowon Koh was supported by Basic Science Research Program through the National Research Foundation of Korea(NRF) funded by the Ministry of Education, Science and Technology(2012R1A1A1001510) }


\begin{document}

\begin{abstract}
For a set $E\subset \mathbb F_q^d$, we define the $k$-resultant magnitude set as $ \Delta_k(E) =\{ \|\textbf{x}_1 + \dots + \textbf{x}_k\|\in \mathbb F_q: \textbf{x}_1, \dots , \textbf{x}_k \in E\},$ where $\|\textbf{v}\|=v_1^2+\cdots+ v_d^2$ for $\textbf{v}=(v_1, \ldots, v_d) \in \mathbb F_q^d.$ In this paper we find a connection between a lower bound of the cardinality of the $k$-resultant magnitude set and the restriction theorem for spheres in finite fields.  As a consequence, it is shown that if $E\subset \mathbb F_q^d$ with $|E|\geq C  q^{\frac{d+1}{2}-\frac{1}{6d+2}},$ then $|\Delta_3(E)|\geq c q$ for $d = 4$ or $d = 6$, and $|\Delta_4(E)| \geq cq$ for even dimensions $d \geq 8.$  In addition, we prove that if $d\geq 8$ is even, and $|E|\geq C_\varepsilon ~q^{\frac{d+1}{2} - \frac{1}{9d -18} + \varepsilon}$ for $\varepsilon >0$, then $|\Delta_3(E)|\geq c q.$

\end{abstract}
\maketitle

\section{Introduction}
Let $\mathbb F_q^d, d\geq 2$, be the $d$-dimensional vector space over a finite field with $q$ elements. Throughout the paper, we assume that the characteristic of $\mathbb F_q$ is not equal to two. For $E \subset \mathbb F_q^d,$ the distance set, denoted by $\Delta_2(E),$ is defined by
\[
\Delta_2(E)=\{\|\textbf{x}-\textbf{y}\|\in \mathbb F_q: \textbf{x},\textbf{y} \in E\},
\]
where $\|\textbf{v}\|=v_1^2+\cdots+ v_d^2$ for $\textbf{v}=(v_1, \ldots, v_d) \in \mathbb F_q^d.$  The Erd\H{o}s-Falconer distance problem in the finite field setting asks for the minimal threshold $\beta$ such that if $|E| \geq C q^{\beta}$ for a sufficiently large constant $C$, then we have $ |\Delta_2(E)|\geq c q$ for some $0<c\leq 1.$  The first distance result was obtained by Bourgain, Katz, and Tao (\cite{BKT}) when $q \equiv 3 \pmod{4}$ is a prime.  Iosevich and Rudnev (\cite{IR07}) studied the general field case, and they obtained the first explicit exponents.  Using discrete Fourier machinery,  they demonstrated that if $E\subset \mathbb F_q^d$ with $|E|\geq Cq^{\frac{d+1}{2}},$ for a sufficiently large constant $C$, then $|\Delta_2(E)|=q.$ \\

The authors in \cite{HIKR10} constructed arithmetic examples which show that the exponent $(d+1)/2$ due to Iosevich and Rudnev is sharp at least in odd dimensions.  Thus, the Erd\H os-Falconer distance problem has been completely resolved in odd dimensions.  On the other hand, it has been conjectured for even dimensions $d\geq 2$ that the exponent $(d+1)/2$ can be improved to the exponent $d/2.$  While this conjecture is open for all even dimensions, the sharp exponent $(d+1)/2$ for odd dimensions was improved for dimension two by the authors in \cite{CEHIK09}.  More precisely they proved that if $E\subset \mathbb F_q^2$ with $|E|\geq C q^{4/3}$ for a sufficiently large constant $C$, then $|\Delta_2(E)|\geq c q$ for some $0<c<1.$  However, the exponent $(d+1)/2$ has not been improved for higher even dimensions $d\geq 4.$  For further discussion on distance problems in finite fields, readers may refer to \cite{Di12, IK08, KP13, KS, KS12, KS13, Sh06, V08}.  See also \cite{Co13, CIP13}, and references contained therein for recent results on the distance problems in the ring setting.\\

The Erd\H{o}s-Falconer distance problem in finite fields can be extended in various directions.  One such direction is as follows.  For each integer $k\geq 2,$ let us consider a function $M_k: (\mathbb F_q^d)^k \to \mathbb F_q.$  Given this function, determine the minimal value $\beta $ such that whenever $E\subset \mathbb F_q^d$ satisfies $|E| \geq C q^{\beta}$ for a sufficiently large constant $C$,  we have $|M_k(E^k)|\geq c q$ for some constant $0<c \leq 1$ independent of $q.$  Note that when $M_2(\textbf{x},\textbf{y})=\|\textbf{x}-\textbf{y}\|$ for $\textbf{x},\textbf{y} \in \mathbb F_q^d,$ we are reduced the Erd\H{o}s-Falconer distance problem in the finite field setting as
\[
\Delta_2(E)=M_2(E\times E)=\{\|\textbf{x}-\textbf{y}\|\in \mathbb F_q: \textbf{x},\textbf{y} \in E\}.
\]

For $k \geq 2$, we will study the function
\[
M_k( \textbf{x}_1, \textbf{x}_2, \ldots, \textbf{x}_k)=\| \textbf{x}_1\pm \textbf{x}_2\pm \cdots \pm \textbf{x}_k\| ~~\mbox{for}~~ \textbf{x}_s \in \mathbb F_q^d,  s=1,2, \ldots, k,
\]
and we denote $M_k(E^k)$ by $\Delta_k(E)$ for $E\subset \mathbb F_q^d.$  Namely, for $E\subset \mathbb F_q^d,$ we define
\[
 \Delta_k(E)=\{ \|\textbf{x}_1\pm \textbf{x}_2 \pm \cdots \pm \textbf{x}_k\| \in \mathbb F_q : \textbf{x}_s \in E,  s=1,2, \ldots, k\}.
 \]
As the choice of signs will be independent of our results, we shall simply define
\[
\Delta_k(E)=\{ \|\textbf{x}_1+ \textbf{x}_2 + \cdots + \textbf{x}_k\| \in \mathbb F_q : \textbf{x}_s \in E,  s=1,2, \ldots, k\}.
\]
Throughout the paper,  the set $\Delta_k(E)$ will be referred to as the $k$-resultant magnitude set.  For brevity, we call $\Delta_2(E)$ the distance set, and when $k = 3$, we simply call $\Delta_3(E)$ the magnitude set.

\begin{question} \label{Q1.1} Let $E\subset \mathbb F_q^d, d\geq 2,$ and $k\geq 2$ be an integer.  Determine the smallest $\beta> 0$ such that if $|E|\ge C q^{\beta}$ with a sufficiently large constant $C>1$, then $|\Delta_k(E)|\ge c q$ for some $0<c \leq 1.$
\end{question}

It is clear that $|\Delta_{k_1}(E)|\leq |\Delta_{k_2}(E)|$ for $2\leq k_1\leq k_2.$  Therefore, as $k$ becomes larger, one might expect a smaller value $\beta$ as the answer to Question \ref{Q1.1}.  However, we conjecture that the answer to Question \ref{Q1.1} is independent of $k$.  For example, if $q=p^2$ for prime $p$ and $E=\mathbb F_p^d,$ then it clearly follows that $|E|=q^{d/2}$ and $|\Delta_k(E)|=\sqrt{q}$ for all $k\geq 2.$  This example says that $\beta$ in Question \ref{Q1.1} can not be smaller than $d/2$ which is the conjectured exponent for the Erd\H{o}s-Falconer distance problem in even dimensions.  This leads us to the following conjecture.  \begin{conjecture}\label{C1.2} Let $E\subset \mathbb F_q^d.$ If $d\geq 2$ is even and $|E|\geq C q^{d/2}$ for a sufficiently large constant $C,$ then for every integer $k\geq 2,$ there exists a constant $0 < c \leq 1$ such that
\[
|\Delta_k(E)| \geq c q.
\]
\end{conjecture}

\subsection{Statement of results}
The techniques used by Iosevich and Rudnev in \cite{IR07} show that if $|E| \geq C q^{\frac{d+1}{2}}$ for a sufficiently large constant $C$, then $|\Delta_k(E)| = q$.  Note that the counterexamples for the Erd\H os-Falconer distance problem immediately show that the exponent $(d+1)/2$ can not be improved in general for odd dimensions.  Thus, we shall only focus on investigating the size of $\Delta_k(E)$  where $E \subset \mathbb{F}_q^d$ is a subset of an even dimensional vector space.  In this paper we demonstrate that the exponent $(d+1)/2$ for the magnitude set can be improved for even $d\geq 4.$  More precisely, we have the following results.
\begin{theorem}\label{T1.3} Let $E\subset \mathbb F_q^d.$ Suppose that $C$ is a sufficiently large constant.\\
\noindent{\it (1)} If  $d=4$ or $6,$ and
 $|E|\geq C q^{\frac{d+1}{2}-\frac{1}{6d+2}},$   then $|\Delta_3(E)|\geq c q$ for some $0 < c \leq 1.$\\
\noindent{\it (2)} If $d\ge 8$ is even and $|E|\geq C q^{\frac{d+1}{2}-\frac{1}{6d+2}},$ then $|\Delta_4(E)|\geq cq$ for some $0 < c \leq 1.$

\end{theorem}

\begin{theorem}\label{T1.4} Suppose that $d\geq 8$ is even and $E\subset \mathbb F_q^d.$
Then given $\varepsilon>0,$ there exists $C_{\varepsilon}>0$ such that
if $|E|\ge C_{\varepsilon}  q^{\frac{d+1}{2} - \frac{1}{9d -18} + \varepsilon},$ then $|\Delta_3(E)|\geq cq$ for some $0 < c \leq 1.$
\end{theorem}

It seems from our results that the exponent $(d+1)/2$ can be improved for the Erd\H{o}s-Falconer distance problem in even dimensions $d\geq 4.$
\begin{remark}\label{Re1.5}  Aside from thinking of the cardinality of $\Delta_3(E)$ as the number of distinct distances of any three vectors in $E \subset \mathbb F_q^d,$ we can also consider it as the number of distinct distances between the origin and the centers of mass of  triangles determined by  $E \subset \mathbb F_q^d$ if $q$ has characteristic greater than $3$.  To see this, notice that if $\textbf{x},\textbf{y},\textbf{z} \in \mathbb F_q^d,$ then $(\textbf{x}+\textbf{y}+\textbf{z})/3$ can be considered as the center of mass of the triangle with vertices $\textbf{x},\textbf{y},\textbf{z}.$
\end{remark}

\subsection{Outline of the paper}
In the remaining parts of the paper, we first provide preliminary lemmas in Section \ref{S2}.  In Section \ref{S3}, we obtain the necessary restriction estimates for spheres.  In the final section, we deduce the formula for $|\Delta_k(E)|$ and we provide the link between the set $\Delta_k(E)$ and the restriction estimates for spheres.

\section{Discrete Fourier analysis and related lemmas}\label{S2}
 As a main technical tool, discrete Fourier analysis plays an important role in proving our results. In this section, we review the basic definitions, and we collect preliminary lemmas which are essential for providing a lower bound for $|\Delta_k(E)|.$

\subsection{Discrete Fourier analysis}
Throughout this paper, $\chi$ denotes a nontrivial additive character of $\mathbb F_q.$ The choice of the character $\chi$ will be independent of the results in this paper.   The orthogonality of the character $\chi$ implies
\[
 \sum_{\textbf{x}\in \mathbb F_q^d} \chi(\textbf{m}\cdot \textbf{x})=\left\{ \begin{array}{ll} 0  \quad&\mbox{if}~~ \textbf{m}\neq (0,\dots,0)\\
q^d  \quad &\mbox{if} ~~\textbf{m}= (0,\dots,0), \end{array}\right.
\]
where  $\textbf{m}\cdot \textbf{x}$ denotes the usual dot-product.  Given a function $g: \mathbb F_q^d \to \mathbb C,$   the Fourier transform of $g$, denoted by $\widetilde{g}$, is defined as
\begin{equation}\label{gtile}
\widetilde{g}(\textbf{x})= \sum_{\textbf{m}\in \mathbb F_q^d} g(\textbf{m}) \chi(-\textbf{x}\cdot \textbf{m})\quad \mbox{for}~~\textbf{x}\in \mathbb F_q^d.
\end{equation}
On the other hand,  if $f: \mathbb F_q^d \rightarrow \mathbb C,$ then  we denote by $\widehat{f}$ the {\bf normalized}  Fourier transform of the function $f$. Thus, we have
\[
 \widehat{f}(\textbf{m})= \frac{1}{q^d} \sum_{\textbf{x}\in \mathbb F_q^d} f(\textbf{x}) \chi(-\textbf{x}\cdot \textbf{m})\quad \mbox{for}~~\textbf{m}\in \mathbb F_q^d.
 \]
We also write $ f^{\vee}(\textbf{m})$ for $\widehat{f}(-\textbf{m}).$ Notice that $\widetilde{ (f^{\vee})}(\textbf{x})=f(\textbf{x})$ for $\textbf{x}\in \mathbb F_q^d.$ Namely,  the Fourier inversion theorem in this content is given by the formula
\[
f(\textbf{x})= \sum_{\textbf{m}\in \mathbb F_q^d}  \widehat{f}(\textbf{m}) \chi(\textbf{m}\cdot \textbf{x}) \quad\mbox{for}~~\textbf{x}\in \mathbb F_q^d.
\]

\begin{remark} Throughout the rest of the article, we will write $E(\textbf{x})$ for the characteristic function (or indicator function) of a set $E\subset \mathbb F_q^d.$
\end{remark}
As a direct application of the orthogonality relation of $\chi,$ it follows that
\[
 \sum_{\textbf{m}\in \mathbb F_q^d} |\widehat{f}(\textbf{m})|^2 = \frac{1}{q^d} \sum_{\textbf{x}\in \mathbb F_q^d} |f(\textbf{x})|^2.
 \]
We refer to this formula as the Plancherel theorem. For example, if we take $f(\textbf{x})=E(x)$, then  the Plancherel theorem implies that
\[
 \sum_{\textbf{m}\in \mathbb F_q^d} |\widehat{E}(\textbf{m})|^2 = \frac{|E|}{q^d}.
 \]
Furthermore,  since $\left|\widehat{E}(\textbf{m})\right| \leq q^{-d}|E|,$ it is clear that for every integer $k\ge 2,$
\begin{equation}\label{R2.1}
\sum_{\textbf{m}\in \mathbb F_q^d} |\widehat{E}(\textbf{m})|^k \le \frac{|E|^{k-2} }{q^{d(k-2)}  } \sum_{\textbf{m}\in \mathbb F_q^d} |\widehat{E}(\textbf{m})|^2  =  \frac{|E|^{k-1}}{q^{dk-d}}.
\end{equation}

We now collect information about the normalized Fourier transform on the sphere.
For $t\in \mathbb F_q,$ the sphere $S_t \subset \mathbb F_q^d$ is defined by
\[
S_t=\{\textbf{x}\in \mathbb F_q^d: x_1^2+\cdots+x_d^2=t\}.
\]
It is well known from Theorem 6.26 and Theorem 6.27 in \cite{LN97} that if  $d\ge3$ and $t \in \mathbb F_q,$ then
\begin{equation}\label{SizeS}
|S_t| = q^{d-1}(1 + o(1)).
\end{equation}

The following result follows immediately from Lemma 4 in \cite{IK10}.
\begin{proposition}\label{P2.1} Let $d\geq 2$ be even and $t\in \mathbb F_q.$ Then, for $\textbf{m}\in \mathbb F_q^d,$
\[  \widehat{S_t}(\textbf{m}) = q^{-1} \delta_0(\textbf{m}) + q^{-d-1} G^d \sum_{\ell \in {\mathbb F}_q^*}
\chi\Big(t\ell+ \frac{\|\textbf{m}\|}{4\ell}\Big),
\]
where $\delta_0(\textbf{m})$ is the delta-function, so that $\delta_0(\textbf{m})=1$ for $\textbf{m}=(0, \ldots, 0)$ and $\delta_0(\textbf{m})=0$ otherwise, and $G$ denotes the Gauss sum
\[
\displaystyle G =\sum_{s\in \mathbb F_q^*} \eta(s) \chi(s) ,
\]
where $\eta$ is the quadratic character of $\mathbb F_q,$ and $\mathbb F_q^* = \mathbb F_q \setminus \{0\}$.  In particular, we have
\[
\widehat{S_0}(\textbf{m})=q^{-1} \delta_0(\textbf{m}) + q^{-d-1} G^d \sum_{\ell \in {\mathbb F}_q^*} \chi( \|\textbf{m}\| \ell)\quad \mbox{for}~~\textbf{m}\in \mathbb F_q^d.
\]
\end{proposition}

\begin{remark}\label{Re2.3} Recall that the Gauss sum satisfies $|G|=\sqrt{q}.$  For $a,b\in \mathbb F_q$, the Kloosterman sum is defined by
\[
 K(a,b) := \sum_{\ell \in \mathbb F_q^*} \chi(a\ell + b/\ell).
 \]
It is well known that $|K(a,b)|\leq 2\sqrt{q}$ for $ab\neq 0.$ For the proof of the Gauss and Kloosterman sum estimation,  see \cite{IK04, LN97}.
\end{remark}

The following result was proved in Proposition 2.2 in \cite{KS13}.
\begin{proposition} \label{P2.2} For $\textbf{m}, \textbf{v} \in \mathbb F_q^d,$ we have
\[
 \sum_{t\in \mathbb F_q} \widehat{S_t}(\textbf{m})~ \overline{\widehat{S_t}}(\textbf{v}) =  q^{-1}\delta_0(\textbf{m})~\delta_0(\textbf{v}) + q^{-d-1} \sum_{s\in\mathbb F_q^*} \chi(s(\|\textbf{m}\|-\|\textbf{v}\|)).
\]
\end{proposition}

\subsection{Evaluation of the counting function $\nu_k$} Let $E\subset \mathbb F_q^d$ and let $k\geq 2$ be an integer. For $t\in \mathbb F_q,$ we define the counting function $\nu_k(t)$ by
\begin{align*}
\nu_k(t):
&=|\{(\textbf{x}_1, \textbf{x}_2, \ldots, \textbf{x}_k)\in E^k: \|\textbf{x}_1+ \textbf{x}_2+\cdots+ \textbf{x}_k\|=t\}|
\\
&=\sum_{\textbf{x}_1,\ldots,\textbf{x}_k\in E} S_t(\textbf{x}_1+\textbf{x}_2+\cdots+\textbf{x}_k).
\end{align*}

Applying the Fourier inversion theorem to $S_t(\textbf{x}_1+\textbf{x}_2+\cdots+\textbf{x}_k),$  it follows from the definition of the normalized Fourier transform that
\begin{equation}\label{R2.2}
\nu_k(t)=q^{dk} \sum_{\textbf{m}\in \mathbb F_q^d} \widehat{S_t}(\textbf{m}) \left(\overline{\widehat{E}(\textbf{m})}\right)^k.
\end{equation}

Then an $L^2$ estimate of $\nu_k$ is as follows.
\begin{lemma}\label{L2.4} Let $E\subset \mathbb F_q^d, d\geq 2.$ Then we have
\[
\sum_{t\in \mathbb F_q} \nu^2_k(t) \leq q^{-1}|E|^{2k} + q^{2dk-d} \sum_{r\in \mathbb F_q} \left| \sum_{\textbf{v}\in S_r} \left(\widehat{E}(\textbf{v})\right)^k \right|^2.\]
\end{lemma}
\begin{proof}
Since $\nu^2_k(t)=\nu_k(t)~ \overline{\nu_k(t)},$  we see from \eqref{R2.2} that
\[
 \sum_{t\in \mathbb F_q} \nu^2_k(t)=q^{2dk} \sum_{\textbf{m}, \textbf{v}\in \mathbb F_q^d}
\left(\overline{\widehat{E}(\textbf{m})}\right)^k  \left(\widehat{E}(\textbf{v})\right)^k
\left( \sum_{t\in \mathbb F_q} \widehat{S_t}(\textbf{m}) \overline{ \widehat{S_t}(\textbf{v})}\right).
\]
From Proposition \ref{P2.2}, we conclude that
\begin{align*}  \sum_{t\in \mathbb F_q} \nu^2_k(t)
&=q^{-1}|E|^{2k} + q^{2dk-d}\hspace{-2mm} \sum_{\substack{\textbf{m}, \textbf{v}\in \mathbb F_q^d:\\ \|\textbf{m}\|=\|\textbf{v}\|}}\hspace{-3mm} \left(\overline{\widehat{E}(\textbf{m})}\right)^k \hspace{-1mm} \left(\widehat{E}(\textbf{v})\right)^k\hspace{-2mm}
-q^{2dk-d-1} \left| \sum_{\textbf{v}\in \mathbb F_q^d}  \left( \widehat{E}(\textbf{v})\right)^k\right|^2\\
&\le q^{-1}|E|^{2k} + q^{2dk-d}\hspace{-2mm} \sum_{\substack{\textbf{m}, \textbf{v}\in \mathbb F_q^d:\\ \|\textbf{m}\|=\|\textbf{v}\|}}\hspace{-3mm} \left(\overline{\widehat{E}(\textbf{m})}\right)^k \hspace{-1mm} \left(\widehat{E}(\textbf{v})\right)^k \\
&=  q^{-1}|E|^{2k} + q^{2dk-d} \sum_{r\in \mathbb F_q} \left| \sum_{\textbf{v}\in S_r} \left(\widehat{E}(\textbf{v})\right)^k \right|^2.\end{align*}
\end{proof}
We need the following lemma.
\begin{lemma}\label{L2.5}
Suppose that  $d\ge 2$ is even and $k\ge 2$ is an integer.
If $E\subset \mathbb F_q^d$ with $|E|\ge 3 q^{d/2},$ then we have
\[
 \left(|E|^k-\nu_k(0)\right)^2 \geq \frac{|E|^{2k}}{9}.
 \]
\end{lemma}
\begin{proof}
Combining \eqref{R2.2} and Proposition \ref{P2.1},  we see that
\begin{align*}\nu_k(0)=&q^{dk}\sum_{\textbf{m}\in \mathbb F_q^d} \left(\overline{\widehat{E}(\textbf{m})}\right)^k
\left(q^{-1}\delta_0(\textbf{m})+q^{-d-1}G^d \sum_{\ell \in \mathbb F_q^*} \chi(\|\textbf{m}\| \ell)\right)\\
=& q^{dk-1} \left(\overline{\widehat{E}(0,\ldots,0)}\right)^k  + q^{dk-d-1} G^d \sum_{\textbf{m}\in \mathbb F_q^d} \left(\overline{\widehat{E}(\textbf{m})}\right)^k \left(   \sum_{\ell \in \mathbb F_q^*} \chi(\|\textbf{m}\| \ell)  \right).
 \end{align*}
Since $\widehat{E}(0,\ldots,0) = q^{-d}|E|,$  we have
\begin{equation}\label{R2.3}
\nu_k(0)=q^{-1}|E|^k +q^{dk-d-1} G^d \sum_{\textbf{m}\in \mathbb F_q^d} \left(\overline{\widehat{E}(\textbf{m})}\right)^k \left(   \sum_{\ell \in \mathbb F_q^*} \chi(\|\textbf{m}\| \ell)  \right).
\end{equation}
Since $\nu_k(0)$ is a nonnegative real number, it is clear that
\[
\nu_k(0) \le q^{-1}|E|^k + q^{dk-d} |G|^d \sum_{\textbf{m}\in \mathbb F_q^d} | \widehat{E}(\textbf{m})|^k.
\]
As $|G|=q^{1/2},$ it follows from \eqref{R2.1} that
\[
 \nu_k(0)\le q^{-1}|E|^k + q^{d/2} |E|^{k-1}.
\]
Since $q\ge 3,$ this clearly implies that if $|E|\ge 3 q^{d/2},$ then
\begin{align*}
|E|^k-\nu_k(0) \ge& |E|^k-q^{-1}|E|^k -q^{d/2} |E|^{k-1}
\\
\ge&\frac{|E|^k}{3} + \left( \frac{|E|^k}{3} -q^{d/2} |E|^{k-1} \right)\ge \frac{|E|^k}{3},
\end{align*}
and the statement of the lemma follows immediately.
\end{proof}

 We shall also use the following result.
\begin{lemma}\label{L2.6} Let $E\subset \mathbb F_q^d.$ Assume that $d\ge 2 $ is even and $k\ge 2$ is an integer. If $|E|\ge  q^{d/2},$ then we have
\[
q^{2dk-d}\left| \sum_{\textbf{m} \in S_0} \left(\widehat{E}(\textbf{m})\right)^k \right|^2 - \nu^2_k(0)\le  4q^{-1}|E|^{2k}.
\]
\end{lemma}

\begin{proof} Observe from \eqref{R2.3} that we can write
\[
\nu_k(0)=q^{-1}|E|^k +q^{dk-d-1} G^d \sum_{\textbf{m}\in \mathbb F_q^d} \left(\overline{\widehat{E}(\textbf{m})}\right)^k \left( -1+  \sum_{\ell \in \mathbb F_q} \chi(\|\textbf{m}\| \ell)  \right).
\]
By the orthogonality relation of $\chi$,  it is easy to see that
\begin{align*}\nu_k(0)=&q^{dk-d}G^d \sum_{\textbf{m}\in S_0} \left(\overline{\widehat{E}(\textbf{m})}\right)^k
+\left( q^{-1}|E|^k-q^{dk-d-1} G^d \sum_{\textbf{m}\in \mathbb F_q^d} \left(\overline{\widehat{E}(\textbf{m})}\right)^k \right)\\
:=& A+B.
\end{align*}
Since $\nu_k(0)\ge 0,$ it follows that
\begin{align*} \nu^2_k(0)=&\nu_k(0) \overline{\nu_k(0)} =( A+B ) (\overline{A} + \overline{B})\\
=& q^{2dk-d} \left|\sum_{\textbf{m}\in S_0} \left( \widehat{E}(\textbf{m})\right)^k\right|^2 + A\overline{B} +\overline{A}B + |B|^2.
\end{align*}
This observation and the definition of $A$ and $B$ yield that
\begin{align*} &q^{2dk-d} \left|\sum_{\textbf{m}\in S_0} \left( \widehat{E}(\textbf{m})\right)^k\right|^2 - \nu^2_k(0) \leq 2|A||B|\\
& \le 2 \left( q^{dk-d/2} \sum_{\textbf{m}\in \mathbb F_q^d} |\widehat{E}(\textbf{m})|^k\right) \left( q^{-1} |E|^k + q^{dk-d/2 -1}\sum_{\textbf{m}\in \mathbb F_q^d} |\widehat{E}(\textbf{m})|^k\right) \\
&\le 2\left(q^{d/2 -1} |E|^{2k-1} +  q^{d-1} |E|^{2k-2}\right),
\end{align*}
where \eqref{R2.1} was applied to obtain the last line.  We complete the proof by observing that  if $|E|\ge  q^{d/2},$ then
\[
\max \left( q^{d/2 -1} |E|^{2k-1}, ~ q^{d-1} |E|^{2k-2} \right) \leq q^{-1}|E|^{2k}.
\]
\end{proof}

\section{Results on the restriction theorem for spheres}\label{S3}
In this section we collect lemmas which can be obtained by applying the restriction theorems for spheres in finite fields. To do this,  we begin by reviewing  the extension problem for spheres.  We denote by $(\mathbb F_q^d, d\textbf{x})$ the $d$-dimensional vector space over $\mathbb F_q$ endowed with the normalized counting measure ``$d\textbf{x}$".  On the other hand,  the dual space of $(\mathbb F_q^d, d\textbf{x})$  will be denoted by $(\mathbb F_q^d, d\textbf{m})$ which we endow  with the counting measure  ``$d\textbf{m}.$"  Notice that both spaces are isomorphic as an abstract group but different measures are given between them. For this reason, the norm of a function depends on its domain. For maximum clarity and ease of exposition of norms, we explicitly define the following norms as sums:  for $1\le s < \infty,$
\begin{align*}
&\|g\|_{L^s(\mathbb F_q^d, d\textbf{m})}^s = \sum_{\textbf{m} \in \mathbb F_q^d} |g(\textbf{m})|^s,\\
& \|f\|_{L^s(\mathbb F_q^d, d\textbf{x})}^s = q^{-d}\,\sum_{\textbf{x} \in \mathbb F_q^d} |f(\textbf{x})|^s.
\end{align*}
In addition, we define
\begin{equation*} \|g\|_{L^\infty(\mathbb F_q^d, d\textbf{m})} =\max_{\textbf{m} \in \mathbb F_q^d} |g(\textbf{m})|.
\end{equation*}

Next, we introduce the normalized surface measures on spheres in finite fields.
For $t\in \mathbb F_q^*,$  we consider a sphere $S_t \subset (\mathbb F_q^d, d\textbf{x}).$
For each $t\in \mathbb F_q^*,$ we endow the sphere $S_t$ with the normalized surface measure $d\sigma$. Thus, if $f: (\mathbb F_q^d, d\textbf{x}) \to \mathbb C,$ then we define
\begin{align*}
&\| f\|_{L^s(S_t,  d\sigma)}^s=\frac{1}{|S_t|} \sum_{\textbf{x}\in S_t} |f(\textbf{x})|^s\quad \mbox{for}~~ 1\le s<\infty,\\
& \| f\|_{L^\infty(S_t,  d\sigma)}= \max_{\textbf{x} \in S_t} |f(\textbf{x})|.\end{align*}
Also recall that if $f:(S_t, d\sigma) \to \mathbb C,$ then the inverse Fourier transform of  $fd\sigma$ is given by
\[
(fd\sigma)^{\vee}(\textbf{m})  =\frac{1}{|S_t|} \sum_{\textbf{x}\in S_t}f(\textbf{x}) \chi(\textbf{m}\cdot \textbf{x})\quad\mbox{for}~~\textbf{m}\in (\mathbb F_q^d, d\textbf{m}).
\]
Since $S_t=-S_t:=\{\textbf{x}\in \mathbb F_q^d: -\textbf{x}\in S_t\}$,  it follows from the definition of the normalized Fourier transform that
\begin{equation}\label{IFT} (d\sigma)^{\vee}(\textbf{m}) = \frac{q^d}{|S_t|} \widehat{S_t}(\textbf{m})\quad\mbox{for}~~\textbf{m}\in (\mathbb F_q^d, d\textbf{m}).\end{equation}

With the above notation,  the extension problem for the sphere $S_t$ is to determine $1\leq p, r\leq \infty$ such that for some $C>0$,
\begin{equation}\label{m1}
\| (fd\sigma)^\vee \|_{L^r(\mathbb F_q^d, d\textbf{m})} \le C \|f\|_{L^p(S_t, d\sigma)} \quad\mbox{for all} ~~
f: S_t \to \mathbb C,
\end{equation}
where the constant $C>0$ may depend on $p, r, d, S_t,$ but it must be  independent of the functions $f$ and the size of the underlying finite field $\mathbb F_q.$
By duality, this extension estimate is the same as the following restriction estimate (see \cite{ MT04,Ta04}) :
\begin{equation}\label{m2}
 \|\widetilde{g}\|_{L^{p^\prime}(S_t, d\sigma)}\leq C \|g\|_{L^{r^\prime}(\mathbb F_q^d, d\textbf{m})}
\quad\mbox{for all} ~~g: \mathbb F_q^d \to \mathbb C,
\end{equation}
where $\widetilde{g}$ is defined as in \eqref{gtile}, and 
$p^{\prime}$ and $r^{\prime}$  denote the H\"{o}lder conjugates of $p$ and $r,$ which mean that $1/p+1/p^{\prime}=1$ and $1/r+1/r^{\prime}=1.$ 
\begin{remark}In this paper, we will use $X \lesssim Y$ to mean that there exists $C>0$, independent of $q$ such that $X \leq CY$, and  we also write $Y\gtrsim X$ for $X \lesssim Y.$  We use $X\sim Y$ to indicate that $\lim_{q \to \infty} X/Y = 1$.  In addition, $X\lessapprox Y$ means that for every $\varepsilon >0$ there exists  $C_{\varepsilon}>0$ such that $ X\le C_{\varepsilon} q^{\epsilon} Y.$
\end{remark}

 By the definition of norms and Fourier transforms, the inequalities in  $\eqref{m1}$ and $\eqref{m2}$ are written in terms of the following sums, respectively:

\begin{equation*} \sum_{\textbf{m}\in \mathbb F_q^d} \left| \frac{1}{|S_t|} \sum_{\textbf{x}\in S_t} \chi(\textbf{m} \cdot \textbf{x}) f(\textbf{x}) \right|^r \leq C^r \frac{1}{|S_t|^{r/p}} \left(\sum_{\textbf{x} \in S_t} |f(\textbf{x})|^p \right)^{r/p}\\
\end{equation*}
and
\begin{equation}\label{m3} \frac{1}{|S_t|} \sum_{\textbf{x}\in S_t} \left| \sum_{\textbf{m}\in \mathbb F_q^d} \chi(-\textbf{m}\cdot \textbf{x}) g(\textbf{m}) \right|^{p'} \le C^{p'} \left(\sum_{\textbf{m}\in \mathbb F_q^d} |g(\textbf{m})|^{r'} \right)^{p'/r'}.
\end{equation}

In particular, in \eqref{m3} if we take $g(x)=E(x)$ for $E\subset \mathbb F_q^d,$ and $p'=k$, then we obtain that

\begin{equation*} \frac{1}{|S_t|} \sum_{\textbf{x} \in S_t} |q^d \widehat{E}(\textbf{x})|^k \lesssim |E|^{k/r'}.
\end{equation*}
Since $|S_t|\sim q^{d-1} $for $t\ne 0,$  if $t\ne 0$, then we can write
\begin{equation*} \sum_{\textbf{x}\in S_t} |\widehat{E}(\textbf{x})|^k \lesssim q^{d-dk-1} |E|^{k/r'}.
\end{equation*}

If we put $ \ell=r'$ and change the variable $\textbf{x}$ into $\textbf{m},$ it follows that
\begin{equation*} \sum_{\textbf{m}\in S_t} |\widehat{E}(\textbf{m})|^k \lesssim q^{d-dk-1} |E|^{k/\ell}.
\end{equation*}
In summary, we obtain the following result.

\begin{lemma}\label{kang1} Assume that the following restriction estimate holds for $1\le k, \ell < \infty$  
 \[\|\widetilde{g}\|_{L^{k}(S_t, d\sigma)}\lesssim \|g\|_{L^{\ell}(\mathbb F_q^d, d\textbf{m})}
\quad\mbox{for all} ~~g: \mathbb F_q^d \to \mathbb C,\, t\in \mathbb F_q^*. \]
Then  we have
\[\max_{t\in \mathbb F_q^*}\left(\sum_{\textbf{m}\in S_t} |\widehat{E}(\textbf{m})|^k \right)\lesssim q^{d-dk-1} |E|^{k/\ell}\quad \mbox{for all }~~E\subset \mathbb F_q^d.\]
\end{lemma}

In the finite field setting, the extension problem for various varieties was first posed by Mockenhaupt and Tao (\cite{MT04}).  They mainly obtained good results for paraboloids in lower dimensions. Their results have been recently improved (see, for example, \cite{IK09, Le13, Le14, LL10}).  The extension problem for spheres is more delicate than that of paraboloids, and it was studied by Iosevich and Koh.  In \cite{IKo08}, they obtained the sharp $L^2-L^4$ extension result for circles, which the authors of \cite{CEHIK09} applied to deduce the exponent $4/3$ for the Erd\H{o}s-Falconer distance problem in dimension two. Recall that if $d=2,$ then the exponent $4/3$ gives a much better result than the exponent $(d+1)/2$ which is optimal for odd dimensions. When $d\geq 3,$ the $L^2-L^{(2d+2)/(d-1)}$ extension result for spheres is also known in \cite{IKo08} and can be also applied to the Erd\H{o}s-Falconer distance problem but we can only obtain the exponent $(d+1)/2.$ \\

In \cite{IK10}, Iosevich and Koh  investigated the $L^p-L^4$ spherical extension problem, and they proved the following result which improves the previous work in \cite{IKo08}.
\begin{proposition}[\cite{IK10}, Theorem 1]\label{P3.1}
Let $d\geq 4$ be even. Then we have
\begin{equation}
\label{R3.1} \|(Ed\sigma)^\vee\|_{L^4(\mathbb E_q^d, d\textbf{m})} \lesssim \|E\|_{L^{(12d-8)/{(9d-12)}}(S_t, d\sigma)} \quad \mbox{for all} ~~E\subset S_t , ~t\neq 0,
\end{equation}
where we identify the set $E$ with the characteristic function of $E\subset S_t.$ 
In addition, using the pigeonhole principle, \eqref{R3.1} implies that
\begin{equation} \label{R3.2} \| (fd\sigma)^\vee\|_{L^4(\mathbb F_q^d, d\textbf{m})} \lessapprox  \|f\|_{L^{{(12d-8)}/{(9d-12)}}(S_t, d\sigma)} \quad \mbox{for all} ~~f: S_t \to \mathbb C,~t\neq 0.
\end{equation}
\end{proposition}

\begin{remark} Theorem 1 in \cite{IK10} is actually the statement \eqref{R3.2}. In order to prove it,  the authors in \cite{IK10}  proved the statement \eqref{R3.1} and then concluded that the statement \eqref{R3.2} holds by just invoking the pigeonhole principle (a dyadic decomposition). It is well known that  \eqref{R3.1} implies \eqref{R3.2} (for example, see the proof of Theorem 17, \cite{Gr03}). In fact, if \eqref{R3.1} is true, then the pigeonhole principle yields that 
\begin{equation*}\| (fd\sigma)^\vee\|_{L^4(\mathbb F_q^d, d\textbf{m})} \lesssim \log{q} \, \|f\|_{L^{{(12d-8)}/{(9d-12)}}(S_t, d\sigma)} \quad \mbox{for all} ~~f: S_t \to \mathbb C,~t\neq 0. \end{equation*}
\end{remark}

Proposition \ref{P3.1} plays an important role in proving results for the cardinality of $\Delta_3(E).$  For the direct application to the problem, we shall use the following  lemma which is actually a corollary of  Proposition \ref{P3.1} .
\begin{lemma}\label{L3.2}
Let $d\geq 4$ be even. If $k> (12d-8)/(3d+4)=4-24/(3d+4),$  then we have
\[\max_{t\in \mathbb F_q^*} \left(\sum_{\textbf{v}\in S_t}\left|\widehat{E}(\textbf{v})\right|^k\right) \lesssim q^{d-dk-1} |E|^{\left( (3k-3)d+4k+2\right) /(3d+4)}
\]
\end{lemma}
\begin{proof} By Lemma \ref{kang1}, it will be enough to show that
\begin{equation} \label{dk1} \|\widetilde{g}\|_{L^k(S_t, d\sigma)} \lesssim \|g\|_{L^{\alpha}(\mathbb F_q^d, d\textbf{m})}\quad \mbox{for all} ~~g: \mathbb F_q^d \to \mathbb C,~t\neq 0, \end{equation}
where $\alpha={k(3d+4)}/{((3k-3)d+4k+2)}.$
It is clear that
\begin{equation}\label{R3.3}\|(fd\sigma)^\vee\|_{L^\infty(\mathbb F_q^d, d\textbf{m})}\leq \|f\|_{L^1(S_t, d\sigma)}\quad\mbox{for all}~~f:S_t \to \mathbb C,~t\neq 0.\end{equation}
For any even integer $d\ge 4,$ recall from \eqref{R3.1} in Proposition \ref{P3.1} that
\begin{equation}\label{R3.4} \|(Ed\sigma)^\vee\|_{L^4(\mathbb F_q^d, d\textbf{m})} \lesssim \|E\|_{L^{{(12d-8)}/{(9d-12)}}(S_t, d\sigma)} \quad \mbox{for all} ~~E\subset S_t, ~t\neq 0.\end{equation}
We need the following proposition which is a direct consequence of Theorem 1.4.19 in \cite{Gr04}.
\begin{proposition}\label{Grafakos} 
Let $0<p_0\ne p_1\le \infty,$ and $0<r_0\ne r_1\le \infty.$ Assume that for some $M_0, M_1>0$, the following estimates hold:
\begin{align*}&\|(Ed\sigma)^\vee \|_{L^{r_0}(\mathbb F_q^d, d\textbf{m} )} \le M_0 \|E\|_{L^{p_0}(S_t, d\sigma)}\\
& \|(Ed\sigma)^\vee \|_{L^{r_1}(\mathbb F_q^d, d\textbf{m} )} \le M_1 \|E\|_{L^{p_1}(S_t, d\sigma)}\end{align*}
for all $E \subset S_t.$ Fix $0<\theta <1$ and let 
\begin{equation*} \frac{1}{p}= \frac{1-\theta}{p_0} +\frac{\theta}{p_1}, \quad \frac{1}{r}= \frac{1-\theta}{r_0} +\frac{\theta}{r_1}, \quad\mbox{and} \quad p\le r.
\end{equation*}
Then there exists a constant $M>0$ such that
\[\|(fd\sigma)^\vee \|_{L^{r}(\mathbb F_q^d, d\textbf{m} )} \le M \|f\|_{L^{p}(S_t, d\sigma)}\quad \mbox{for all} ~~f: S_t \to \mathbb C,
\]
where $M>0$ is independent of the functions $f$ and $q$, the size of the underlying finite field $\mathbb F_q.$
\end{proposition}

Since we assume that $k> (12d-8)/(3d+4)$ and $d\ge 4,$  it is easy to see that $1< k/(k-1)<(12d-8)/(9d-12).$ Therefore, applying Proposition \ref{Grafakos} with \eqref{R3.3} and \eqref{R3.4}, we see that
\[
\|(fd\sigma)^\vee\|_{L^{{k(3d+4)}/{(3d-2)}}(\mathbb F_q^d, d\textbf{m})}\lesssim \|f\|_{L^{{k}/{(k-1)}}(S_t, d\sigma)}\quad\mbox{for all}~~f:S_t \to \mathbb C,~t\neq 0.
\]
By duality\footnote{This means that the inequality \eqref{m1} is equivalent to the inequality \eqref{m2}}, the statement \eqref{dk1} follows immediately and we complete the proof of Lemma \ref{L3.2}.
\end{proof}

Observe that the hypotheses of Lemma \ref{L3.2} are satisfied if $k\ge 4$ and $d\ge 4$ is even  or if $k=3$ and $d=4$ or $6.$   However, in the case when $k=3$ and $d\ge 8 $ is even,  it is clear that Lemma \ref{L3.2} is not applicable.  In this case, we shall alternatively use the following result.
\begin{lemma}\label{L3.3}
Let $d\ge 8$ be an even integer. If $E\subset \mathbb F_q^d$ and $|E|\ge q^{(d-1)/2},$ then we have

\[\max_{t\in \mathbb F_q^*} \left(\sum_{\textbf{v}\in S_t}\left|\widehat{E}(\textbf{v})\right|^3\right) \lessapprox q^{(-27d^2+75d+12)/(12d-32)} |E|^{(15d-46)/(6d-16)}.\]
\end{lemma}
\begin{proof}
Since $|S_t|\sim q^{d-1}$ for even $d\ge 8,$ and $\widehat{E}(\textbf{v}) =q^{-d}\widetilde{E}(\textbf{v}),$ it suffices to show from the definition of norms that 
if $E\subset \mathbb F_q^d$ and $|E|\ge q^{(d-1)/2},$ then 
\begin{equation}\label{stupidestimate}
 \|\widetilde{E}\|_{L^3(S_t, d\sigma)} \lessapprox  q^{{(-3d^2+23d-20)}/{(36d-96)}} |E|^{{(15d-46)}/{(18d-48)}}\quad\mbox{for all}~~t\neq 0 .
 \end{equation}
 Let us assume for a moment that if $t\in \mathbb F_q^*,$ then
\begin{equation}\label{R3.5}
\|\widetilde{E}\|_{L^2(S_t, d\sigma)} \lesssim q^{{(-d+1)}/{4}} |E| \quad\mbox{for all} ~~E\subset \mathbb F_q^d \quad\mbox{with}~~|E|\geq q^{{(d-1)}/{2}}.
\end{equation}
By duality, \eqref{R3.2} in Proposition \ref{P3.1} implies that
\[
 \|\widetilde{g}\|_{L^{{(12d-8)}/{(3d+4)}}(S_t, d\sigma)} \lessapprox \|g\|_{L^{{4}/{3}}(\mathbb F_q^d, d\textbf{m})}\quad \mbox{for all} ~~g: \mathbb F_q^d \to \mathbb C,~t\neq 0.
 \]
Taking $g$ as a characteristic function on $E\subset \mathbb F_q^d,$  we obtain that
\begin{equation}\label{R3.6}
\|\widetilde{E}\|_{L^{{(12d-8)}/{(3d+4)}}(S_t, d\sigma)} \lessapprox \|E\|_{L^{{4}/{3}}(\mathbb F_q^d, d\textbf{m})}=|E|^{{3}/{4}}\quad \mbox{for all} ~~E\subset \mathbb F_q^d, ~t\neq 0.
\end{equation}
Since $2<3< (12d-8)/(3d+4)$ for $d\ge 8$,  we are able to interpolate \eqref{R3.5} and \eqref{R3.6} for $E\subset \mathbb F_q^d$ with $|E|\ge q^{(d-1)/2}$ so that 
the inequality \eqref{stupidestimate} will be established. For the readers' convenience, we shall show how to deduce the inequality \eqref{stupidestimate} from inequalities \eqref{R3.5} and \eqref{R3.6}. Let $0 < \theta=(6d-4)/(9d-24) <1 $ for even $d\ge 8.$ Observe that 
\begin{equation}\label{HC} 
\frac{1}{3}= \frac{1-\theta}{2} +\frac{(3d+4)\theta}{12d-8}.
\end{equation}
By  H\"{o}lder's inequality (see \cite{Jo}) and the definition of norms, it follows 
\begin{align*}\|\widetilde{E}\|_{L^3(S_t, d\sigma)} 
&=\|\widetilde{E}^{(1-\theta)}\,\widetilde{E}^\theta\|_{L^3(S_t, d\sigma)}\\
&\le \|\widetilde{E}^{(1-\theta)}\|_{L^{2/(1-\theta)}(S_t, d\sigma)}\,
\|\widetilde{E}^\theta\|_{L^{(12d-8)/[(3d+4)\theta]}(S_t, d\sigma)}\\
&= \|\widetilde{E}\|_{L^2(S_t, d\sigma)}^{1-\theta}\, \|\widetilde{E}\|_{L^{(12d-8)/(3d+4)}(S_t, d\sigma)}^\theta  
\end{align*}
From \eqref{R3.5}, \eqref{R3.6}, and the definition of $\theta$, we conclude that
\begin{align*}\|\widetilde{E}\|_{L^3(S_t, d\sigma)} &\lessapprox 
\left(q^{{(-d+1)}/{4}} |E|\right)^{1-\theta}  |E|^{ 3\theta/4} \\
&=q^{{(-3d^2+23d-20)}/{(36d-96)}} |E|^{{(15d-46)}/{(18d-48)}}.
\end{align*}

To complete the proof of Lemma \ref{L3.3}, it therefore remains to prove
\eqref{R3.5}. Now we prove \eqref{R3.5}. 
Since $|S_t|\sim q^{d-1},$ by the definition of norms, the proof of \eqref{R3.5} amounts to showing that if $t\in \mathbb F_q^*$ and $E\subset \mathbb F_q^d$ with $|E|\ge q^{\frac{d-1}{2}},$ then
\begin{equation}\label{realaim}
\sum_{\textbf{x} \in S_t} |\widetilde{E}(\textbf{x})|^2 \lesssim q^{\frac{d-1}{2}} |E|^2.
\end{equation}

It follows from the definition of the Fourier transforms that 
\begin{align*} \sum_{\textbf{x} \in S_t} |\widetilde{E}(\textbf{x})|^2 &= \sum_{\textbf{x} \in S_t} \sum_{ \textbf{m}, \textbf{m}' \in E} \chi(-\textbf{x} \cdot (\textbf{m} -\textbf{m}')) =\sum_{ \textbf{m}, \textbf{m}'\in E} q^d \widehat{S_t}(\textbf{m} -\textbf{m}')\\ 
&= q^d |E| \widehat{S_t}(0, \ldots, 0) + \sum_{\textbf{m}, \textbf{m}' \in E: \textbf{m}\ne \textbf{m}'} q^d \widehat{S_t}(\textbf{m} -\textbf{m}')\\
&\le |E||S_t| +  \left(\max_{\textbf{n} \in \mathbb F_q^d\setminus \{(0, \ldots, 0)\}} |\widehat{S_t}(\textbf{n})| \right) \sum_{\textbf{m}, \textbf{m}' \in E:\textbf{m}\ne \textbf{m}'} q^d\\
&\lesssim |E|q^{d-1} + |E|^2 q^d \left(\max_{\textbf{n} \in \mathbb F_q^d\setminus \{(0, \ldots, 0)\}} |\widehat{S_t}(\textbf{n})| \right).\end{align*}
Now, observe from Proposition \ref{P2.1} that if $t\ne 0$, then 
\[\left(\max_{\textbf{n} \in \mathbb F_q^d\setminus \{(0, \ldots, 0)\}} |\widehat{S_t}(\textbf{n})| \right) \lesssim q^{-\frac{d+1}{2}}.\]
Thus, we conclude that
\begin{equation*}\sum_{\textbf{x} \in S_t} |\widetilde{E}(\textbf{x})|^2 \lesssim |E|q^{d-1} + q^{\frac{d-1}{2}} |E|^2 \lesssim q^{\frac{d-1}{2}} |E|^2, \end{equation*}
 where the last inequality follows by our assumption that $|E|\ge q^{\frac{d-1}{2}}.$

\end{proof}
\section{Proofs of main theorems (Theorems \ref{T1.3} and \ref{T1.4})}\label{S4}
We begin by deriving the formula for a lower bound of $|\Delta_k(E)|.$
Let $E\subset \mathbb F_q^d$ and let $k\geq 2$ be an integer.
For $t\in \mathbb F_q,$  recall that  the counting function $\nu_k(t)$ is defined by
\[
\nu_k(t)=|\{(\textbf{x}_1, \textbf{x}_2, \ldots, \textbf{x}_k)\in E^k: \|\textbf{x}_1+ \textbf{x}_2+\cdots+ \textbf{x}_k\|=t\}|.
\]
Also recall that the $k$-resultant magnitude set $\Delta_k(E)$ is given by
\[
\Delta_k(E)=\{ \|\textbf{x}_1+ \textbf{x}_2 + \cdots + \textbf{x}_k\| \in \mathbb F_q : \textbf{x}_s \in E,  s=1,2, \ldots, k\}.
\]
Notice that  $ \nu_k(t) \neq 0 \iff t\in \Delta_k(E).$ It is clear that
\[
 |E|^k-\nu_k(0) =\sum_{t\in \mathbb F_q^*} \nu_k(t).
 \]
Squaring both sizes and using the Cauchy-Schwarz inequality, we see that
\[
 ( |E|^k-\nu_k(0))^2 \leq |\Delta_k(E)| \sum_{t\in \mathbb F_q^*} \nu^2_k(t).
 \]
Namely, we obtain that
\begin{equation}\label{R4.1}
|\Delta_k(E)|\ge \frac{ ( |E|^k-\nu_k(0))^2}{\sum_{t\in \mathbb F_q^*} \nu^2_k(t)}.
\end{equation}

\begin{lemma}\label{L4.1} Let $E\subset \mathbb F_q^d.$  Suppose that $d\geq 2$ is even and $k\ge 2$ is an integer.
If $|E|\ge 3q^{d/2},$ then we have
\[
 |\Delta_k(E)| \gtrsim \min \left(q,~~\frac{|E|^{k+1}}{ q^{dk} \max_{r\in \mathbb F_q^*} \left(\sum_{\textbf{v} \in S_r} |\widehat{E}(\textbf{v})|^{k}\right)  }\right) . \]
\end{lemma}

\begin{proof}
First, we find an upper bound for $\sum_{t\in \mathbb F_q^*} \nu^2_k(E).$ Write
\[
\sum_{t\in \mathbb F_q^*} \nu^2_k(t)=\left(\sum_{t\in \mathbb F_q} \nu^2_k(t)\right)- \nu^2_k(0).
\]
From Lemma \ref{L2.4} and Lemma \ref{L2.6},  we see that
\begin{align*}\sum_{t\in \mathbb F_q^*} \nu^2_k(E) 
\leq& \, q^{-1}|E|^{2k} + q^{2dk-d} \sum_{r\in \mathbb F_q} \left| \sum_{\textbf{v}\in S_r}\, \left(\widehat{E}(\textbf{v})\right)^k \right|^2 - \nu^2_k(0)\\
\leq&\,  5q^{-1}|E|^{2k} + q^{2dk-d} \sum_{r\in \mathbb F_q^*} \left| \sum_{\textbf{v}\in S_r} \left(\widehat{E}(\textbf{v})\right)^k \right|^2\\
\lesssim& \,q^{-1}|E|^{2k} + q^{2dk-d} \sum_{r\in \mathbb F_q^*} \left(\sum_{\textbf{v}\in S_r} \left|\widehat{E}(\textbf{v})\right|^k\right)^2. \end{align*}
Since $S_i$ and $S_j$ are disjoint for $i\ne j$, and $\displaystyle \bigcup_{r\in \mathbb F_q} S_r = \mathbb F_q^d,$ it follows that
\[ \sum_{t\in \mathbb F_q^*} \nu^2_k(E) \lesssim q^{-1}|E|^{2k} + q^{2dk-d} \left[\max_{r\in \mathbb F_q^*} \left(\sum_{\textbf{v}\in S_r}\left|\widehat{E}(\textbf{v})\right|^k\right)\right] \sum_{\textbf{v} \in \mathbb F_q^d}
\left|\widehat{E}(\textbf{v})\right|^k
\]
Now, using \eqref{R2.1},  we obtain that
\begin{equation}\label{R4.2} \sum_{t\in \mathbb F_q^*} \nu^2_k(E) \lesssim q^{-1}|E|^{2k} + q^{dk} |E|^{k-1}\left[\max_{r\in \mathbb F_q^*} \left(\sum_{\textbf{v}\in S_r}\left|\widehat{E}(\textbf{v})\right|^k\right)\right].
\end{equation}

Since  it follows from Lemma \ref{L2.5} that $\left(|E|^k-\nu_k(0)\right)^2 \geq \dfrac{|E|^{2k}}{9},$
combining \eqref{R4.1} with \eqref{R4.2} yields that
\[
|\Delta_k(E)| \gtrsim \frac{|E|^{2k}}{q^{-1}|E|^{2k} + q^{dk} |E|^{k-1}\left[\max_{r\in \mathbb F_q^*} \left(\sum_{\textbf{v}\in S_r}\left|\widehat{E}(\textbf{v})\right|^k\right)\right] }.
\]
This implies the conclusion of Lemma \ref{L4.1} and completes the proof.
\end{proof}

We are ready to prove our main results.

\subsection{ Proof of  Theorem \ref{T1.3}}
In this subsection, we restate Theorem \ref{T1.3} and provide a complete proof.  The statement of Theorem \ref{T1.3} will be a direct consequence from Lemma \ref{L4.1} and Lemma \ref{L3.2}.\\

{\it \noindent{\bf Theorem \ref{T1.3}.}  Let $E\subset \mathbb F_q^d.$ Suppose that $C>1$ is a sufficiently large constant independent of $q.$\\
\noindent{\it (1)} If  $d=4$ or $6,$ and
 $|E|\geq C q^{\frac{d+1}{2}-\frac{1}{6d+2}},$   then $|\Delta_3(E)|\gtrsim q.$\\
\noindent{\it (2)} If $d\ge 8$ is even and $|E|\geq C q^{\frac{d+1}{2}-\frac{1}{6d+2}},$ then $|\Delta_4(E)|\gtrsim q.$}
\begin{proof} We shall prove the statements {\it (1)} and {\it (2)} of Theorem 1.3 at one time.
To the end, notice that if we take $k=3$ for $d=4$ or $6,$ or if we choose  $k=4$ for $d\ge 8$ even, then $k> (12d-8)/(3d+4)$ which is the hypothesis of Lemma \ref{L3.2}.
In either case, we can therefore use the conclusion of Lemma \ref{L3.2}.
Thus, combining Lemma \ref{L4.1} and Lemma \ref{L3.2} yields that
\begin{equation} |\Delta_k(E)| \gtrsim \min \left(q,~~\frac{|E|^{k+1}}{ q^{d-1} |E|^{((3k-3)d+4k+2)/(3d+4)}} \right).
\end{equation}

By a direct computation,  this implies that there exists a large constant $C>1$ such that  if $|E|\ge C q^{(3d^2+4d)/(6d+2)} = Cq^{\frac{d+1}{2} - \frac{1}{6d+2}},$   then
$|\Delta_k(E)|\gtrsim q.$ Thus, the proof is complete.
\end{proof}

\subsection{ Proof of  Theorem \ref{T1.4}}
The proof of Theorem \ref{T1.4} can be completed by applying Lemma \ref{L4.1} and Lemma \ref{L3.3}.
Here, we restate Theorem \ref{T1.4} and provide a complete proof.\\

\noindent{\it {\bf Theorem \ref{T1.4}}   Suppose that $d\geq 8$ is even and $E\subset \mathbb F_q^d.$
Then given $\varepsilon>0,$ there exists $C_{\varepsilon}>0$ such that
if $|E|\ge C_{\varepsilon}  q^{\frac{d+1}{2} - \frac{1}{9d -18} + \varepsilon}  ,$ then $|\Delta_3(E)|\gtrsim q.$}

\begin{proof} Suppose that $d\ge 8$ is even and $E\subset \mathbb F_q^d$ with $|E|\ge 3 q^{d/2}.$    Then Lemma \ref{L4.1} with $k=3$ yields
\begin{equation}\label{R4.4} |\Delta_3(E)| \gtrsim \min \left(q,~~\frac{|E|^4}{ q^{3d} \max_{t\in \mathbb F_q^*} \left(\sum_{\textbf{v} \in S_t} |\widehat{E}(\textbf{v})|^{3}\right)} \right). 
\end{equation}
Recall from  Lemma \ref{L3.3} that
\[\max_{t\in \mathbb F_q^*} \left(\sum_{\textbf{v}\in S_t}\left|\widehat{E}(\textbf{v})\right|^3\right) \lessapprox q^{(-27d^2+75d+12)/(12d-32)} |E|^{(15d-46)/(6d-16)}.\]

Given $\varepsilon >0,$ let $\delta=\varepsilon (9d-18)/(6d-16)>0.$  Choose $C_\delta>0$ such that
\[\max_{t\in \mathbb F_q^*} \left(\sum_{\textbf{v}\in S_t}\left|\widehat{E}(\textbf{v})\right|^3\right) \leq C_{\delta} q^\delta q^{(-27d^2+75d+12)/(12d-32)} |E|^{(15d-46)/(6d-16)}.\]

It follows from this inequality and \eqref{R4.4} that if $ |E|\ge 3 q^{d/2}$, then
\[
|\Delta_3(E)| \gtrsim \min \left(q,~~\frac{|E|^{4}}{ q^{3d}C_{\delta} q^\delta q^{(-27d^2+75d+12)/(12d-32)} |E|^{(15d-46)/(6d-16)}} \right).
 \]
We may assume that $C_\delta>0$ is a sufficiently large constant. Thus, a direction calculation shows that if
\[
|E|\geq C_\delta^{(6d-16)/(9d-18)} q^{\delta (6d-16)/(9d-18)} q^{(9d^2-9d-20)/(18d-36)},
\]
then we have $|\Delta_3(E)|\gtrsim q.$ Letting $C_\epsilon=C_\delta^{(6d-16)/(9d-18)},$ we conclude that if
\[
|E|\ge C_\varepsilon q^{\varepsilon} q^{(9d^2-9d-20)/(18d-36)} =C_\varepsilon q^{\frac{d+1}{2} - \frac{1}{9d-18} + \varepsilon},
\]
then  $|\Delta_3(E)|\gtrsim q.$ This completes the proof.

\end{proof}

{\bf Acknowledgement :} The authors would like to thank the referee for his/her valuable comments which enabled us to complete the final version of  this paper.

\end{document}